\documentclass{amsart}
\usepackage{amssymb}
\usepackage{amsmath}
\usepackage{amsfonts}

\newtheorem{theorem}{Theorem}
\theoremstyle{plain}

\newtheorem{proposition}{Proposition}

\numberwithin{equation}{section}
\input{tcilatex}

\begin{document}
\title[Some results on a cross-section]{Some results on a cross-section\\
in the tensor bundle}
\author{A. GEZER}
\address{Ataturk University, Faculty of Science, Department of Mathematics,
25240, Erzurum-Turkey.}
\email{agezer@atauni.edu.tr}
\author{M. ALTUNBAS}
\address{Erzincan University, Faculty of Science and Art, Department of
Mathematics, 24030, Erzincan-Turkey.}
\email{maltunbas@erzincan.edu.tr}

\begin{abstract}
The present paper is devoted to some results concerning with the complete
lifts of an almost complex structure and a connection in a manifold to its
(0,q)-tensor bundle along the corresponding cross-section.
\end{abstract}

\maketitle

\section{\protect\bigskip \textbf{Introduction}}

The behaviour of the lifts of tensor fields and connections on a manifold to
its different bundles along the corresponding cross-sections are studied by
several authors. For the case tangent and cotangent bundles, see \cite%
{YanoIshihara:DiffGeo,Yano,Yano2} and also tangent bundles of order $2$ and
order $r$, see \cite{Tani,Houh}. In \cite{Gezer2}, the first author and his
collaborator studied the complete lift of an almost complex structure in a
manifold on the so-called pure cross-section of its $(p,q)$-tensor bundle by
means of the Tachibana operator (for diagonal lift to the $(p,q)$-tensor
bundle see \cite{Gezer1} and for the $(0,q)$-tensor bundle see \cite{Magden}%
). Moreover they proved that if a manifold admits an almost complex
structure, then so does on the pure cross-section of its $(p,q)$- tensor
bundle provided that the almost complex structure is integrable. In \cite%
{Magden1}, the authors give detailed description of geodesics of the $(p,q)$%
- tensor bundle with respect to the complete lift of an affine connection.

The purpose of the present paper is two-fold. Firstly, to show the complete
lift of an almost complex structure in a manifold to its $(0,q)$-tensor
bundle along the corresponding cross-section, when restricted to the
cross-section determined by an almost analytic tensor field, is an almost
complex structure. Finally, to study the behaviour of the complete lift of a
connection on the cross-section of the $(0,q)$-tensor bundle.

Throughout this paper, all manifolds, tensor fields and connections are
always assumed to be differentiable of class $C^{\infty }$. Also, we denote
by $\Im _{q}^{p}(M)$ the set of all tensor fields of type $(p,q)$ on $M$,
and by $\Im _{q}^{p}(T_{q}^{0}(M))$ the corresponding set on the $(0,q)$%
-tensor bundle $T_{q}^{0}(M)$. The Einstein summation convention is used,
the range of the indices $i,j,s$ being always $\{1,2,...,n\}$.

\section{Preliminaries}

Let $M$ be a differentiable manifold of class $C^{\infty }$ and finite
dimension $n$. Then the set $T_{q}^{0}(M)=\bigcup_{P\in M}T_{q}^{0}(P)$, $%
q>0,$ is the tensor bundle of type $(0,q)$ over $M$, where $\bigcup $
denotes the disjoint union of the tensor spaces $T_{q}^{0}(P)$ for all $P\in
M$. For any point $\tilde{P}$ of $T_{q}^{0}(M)$ such that $\tilde{P}\in
T_{q}^{0}(M)$, the surjective correspondence $\tilde{P}\rightarrow P$
determines the natural projection $\pi :T_{q}^{0}(M)\rightarrow M$. The
projection $\pi $ defines the natural differentiable manifold structure of $%
T_{q}^{0}(M)$, that is, $T_{q}^{0}(M)$ is a $C^{\infty }$-manifold of
dimension $n+n^{q}$. If $x^{j}$ are local coordinates in a neighborhood $U$
of $P\in M$, then a tensor $t$ at $P$ which is an element of $T_{q}^{0}(M)$
is expressible in the form $(x^{j},t_{j_{1}...j_{q}})$, where $%
t_{j_{1}...j_{q}}$ are components of $t$ with respect to natural base. We
may consider $(x^{j},t_{j_{1}...j_{q}})=(x^{j},x^{\bar{j}})=x^{J}$, $%
j=1,...,n$, $\bar{j}=n+1,...,n+n^{q}$, $J=1,...,n+n^{p+q}$ as local
coordinates in a neighborhood $\pi ^{-1}(U)$.

Let $V=V^{i}\frac{\partial }{\partial x^{i}}$ and $A={A_{j_{1}...j_{q}}}%
dx^{j_{1}}\otimes \cdots \otimes dx^{j_{q}}$ be the local expressions in $U$
of a vector field $V$ \ and a $(0,q)-$tensor field $A$ on $M$, respectively.
Then the vertical lift $^{V}A$ of $A$ and the complete lift $^{C}V$ of $V$
are given, with respect to the induced coordinates, by\noindent 

\begin{equation}
{}^{V}A=\left( 
\begin{array}{c}
{0} \\ 
{A_{j_{1}...j_{q}}}%
\end{array}%
\right)  \label{G2.1}
\end{equation}%
and

\begin{equation}
{}^{C}V=\left( 
\begin{array}{c}
{V^{j}} \\ 
-\dsum\limits_{\lambda =1}^{q}{t_{j_{1}...m...j_{q}}\partial _{j_{\lambda
}}V^{m}}%
\end{array}%
\right) .  \label{G2.2}
\end{equation}

Suppose that there is given a tensor field $\xi \in \Im _{q}^{0}(M)$. Then
the correspondence $x\mapsto \xi _{x},$ $\xi _{x}$ being the value of $\xi $
at $x\in M$, determines a mapping $\sigma _{\xi }:M\mapsto T_{q}^{0}(M),$
such that $\pi \circ \sigma _{\xi }=id_{M},$ and the $n$ dimensional
submanifold $\sigma _{\xi }(M)$ of $T_{q}^{0}(M)$ is called the
cross-section determined by $\xi $. If the tensor field $\xi $ has the local
components $\xi _{k_{1}\cdots k_{q}}(x^{k})$, the cross-section $\sigma
_{\xi }(M)$ is locally expressed by%
\begin{equation}
\left\{ 
\begin{array}{l}
x^{k}=x^{k}, \\ 
x^{\overline{k}}=\xi _{k_{1}\cdots k_{q}}(x^{k})%
\end{array}%
\right.  \label{G2.3}
\end{equation}%
with respect to the coordinates $(x^{k},x^{\overline{k}})$ in $T_{q}^{0}(M)$%
. Differentiating (\ref{G2.3}) by $x^{j}$, we see that $n$ tangent vector
fields $B_{j}$ to $\sigma _{\xi }(M)$ have components\hspace{5mm}%
\begin{equation}
(B_{j}^{K})=(\dfrac{\partial x^{K}}{\partial x^{j}})=\left( 
\begin{array}{c}
\delta _{j}^{k} \\ 
\partial _{j}\xi _{k_{1}\cdots k_{q}}%
\end{array}%
\right)  \label{G2.4}
\end{equation}%
with respect to the natural frame $\left\{ {\partial _{k},\partial _{%
\overline{k}}}\right\} $ in $T_{q}^{0}(M)$.

On the other hand, the fibre is locally expressed by%
\begin{equation*}
\left\{ 
\begin{array}{l}
x^{k}=const., \\ 
t_{k_{1}\cdots k_{q}}=t_{k_{1}\cdots k_{q}},%
\end{array}%
\right.
\end{equation*}%
$t_{k_{1}\cdots k_{q}}$ being considered as parameters. Thus, on
differentiating with respect to $x^{\overline{j}}=t_{j_{1}\cdots j_{q}},$ we
see that $n^{q}$ tangent vector fields $C_{\overline{j}}$ to the fibre have
components \hspace{5mm}%
\begin{equation}
(C_{\overline{j}}^{K})=(\dfrac{\partial x^{K}}{\partial x^{\overline{j}}}%
)=\left( 
\begin{array}{c}
0 \\ 
\delta _{k_{1}}^{j_{1}}\cdots \delta _{k_{q}}^{j_{q}}%
\end{array}%
\right)  \label{G2.5}
\end{equation}%
with respect to the natural frame $\left\{ {\partial _{k},\partial _{%
\overline{k}}}\right\} $ in $T_{q}^{0}(M)$.

We consider in $\pi ^{-1}(U)\subset T_{q}^{0}(M),$\ $n+n^{q\text{ }}$local
vector fields $B_{j}$ and $C_{\overline{j}}$ along $\sigma _{\xi }(M)$. They
form a local family of frames $\left[ {B_{j},C_{\overline{j}}}\right] $
along $\sigma _{\xi }(M)$, which is called the adapted $(B,C)-$frame of $%
\sigma _{\xi }(M)$ in $\pi ^{-1}(U)$. Taking account of (\ref{G2.2}) on the
cross-section $\sigma _{\xi }(M)$, and also (\ref{G2.4}) and (\ref{G2.5}),
we can easily prove that, the complete lift $^{C}V$ has along $\sigma _{\xi
}(M)$ components of the form%
\begin{equation}
^{C}V=\left( 
\begin{array}{c}
V^{j} \\ 
-L_{V}\xi _{j_{1}\cdots j_{q}}%
\end{array}%
\right)  \label{G2.6}
\end{equation}%
with respect to the adapted $(B,C)$-frame. From (\ref{G2.1}), (\ref{G2.4})
and (\ref{G2.5}), the vertical lift ${}^{V}A$ also has components of the form

\begin{equation}
{}^{V}A=\left( 
\begin{array}{c}
{0} \\ 
{A_{j_{1}...j_{q}}}%
\end{array}%
\right)  \label{G2.7}
\end{equation}%
with respect to the adapted $(B,C)$- frame.

\section{Almost complex structures on a pure cross-section in the $(0,q)$%
-tensor bundle}

\bigskip A tensor field $\xi \in \Im _{q}^{0}(M)$ is called pure with
respect to $\varphi \in \Im _{1}^{1}(M)$, if \cite%
{Gezer2,Koto,Magden,Muto,Gezer,Salimov,Tachibana,YanoAko:1968}:%
\begin{equation}
\varphi _{j_{1}}^{r}\xi _{r\cdots j_{q}}=\cdots =\varphi _{j_{q}}^{r}\xi
_{j_{1}\cdots r}=\overset{\ast }{\xi }_{j_{1}\cdots j_{q}}.  \label{G3.1}
\end{equation}%
In particular, vector and covector fields will be considered to be pure.

Let $\overset{\ast }{\Im }$\ $_{q}^{0}(M)$ denotes a module of all the
tensor fields $\xi \in \Im _{q}^{0}(M)$ which are pure with respect to $%
\varphi $. Now, we consider a pure cross-section $\sigma _{\xi }^{\varphi
}(M)$ determined by $\xi \in \overset{\ast }{\Im }\ _{q}^{0}(M)$. The
complete lift $^{C}\varphi $ of $\varphi $ along the pure cross-section $%
\sigma _{\xi }^{\varphi }(M)$ to $T_{q}^{0}(M)$ has local components of the
form

\begin{equation*}
^{C}\varphi =\left( 
\begin{array}{cc}
\varphi _{l}^{k} & 0 \\ 
-(\Phi _{\varphi }\xi )_{lk_{1}...k_{q}} & \varphi _{k_{1}}^{r_{1}}\delta
_{k_{2}}^{r_{2}}...\delta _{k_{q}}^{r_{q}}%
\end{array}%
\right)
\end{equation*}%
with respect to the adapted $(B,C)-$frame of $\sigma _{\xi }^{\varphi }(M)$,
where $(\Phi _{\varphi }\xi )_{lk_{1}\cdots k_{q}}=\varphi _{l}^{m}\partial
_{m}\xi _{k_{1}\cdots k_{q}}-\partial _{l}\overset{\ast }{\xi }_{k_{1}\cdots
k_{q}}+\sum\limits_{a=1}^{q}(\partial _{k_{a}}\varphi _{l}^{m})\xi
_{k_{1}\cdots m\cdots k_{q}}$ is the Tachibana operator.

We consider that the local vector fields \hspace{5mm}%
\begin{equation*}
^{C}X_{(i)}=^{C}(\dfrac{\partial }{\partial x^{i}})=^{C}(\delta _{i}^{h}%
\dfrac{\partial }{\partial x^{h}})=\left( 
\begin{array}{c}
\delta _{i}^{h} \\ 
0%
\end{array}%
\right)
\end{equation*}%
and%
\begin{equation*}
^{V}X^{(\overset{-}{i})}=^{V}(dx^{i_{1}}\otimes \cdots \otimes
dx^{i_{q}})=^{V}(\delta _{h_{1}}^{i_{1}}\cdots \delta
_{h_{q}}^{i_{q}}dx^{h_{1}}\otimes \cdots \otimes dx^{h_{q}})=\left( 
\begin{array}{c}
0 \\ 
\delta _{h_{1}}^{i_{1}}\cdots \delta _{h_{q}}^{i_{q}}%
\end{array}%
\right)
\end{equation*}%
$i=1,...,n,\overline{i}=n+1,...,n+n^{q}$ span the module of vector fields in 
$\pi ^{-1}(U)$. Hence, any tensor fields is determined in $\pi ^{-1}(U)$ by
their actions on $^{C}V$ and $^{V}A$ for any $V\in \Im _{0}^{1}(M)$ and $%
A\in \Im _{q}^{0}(M)$. The complete lift $^{C}\varphi $ along the pure
cross-section $\sigma _{\xi }^{\varphi }(M)$ has the properties%
\begin{equation}
\left\{ 
\begin{array}{l}
^{C}\varphi (^{C}V)=^{C}(\varphi (V))+^{V}((L_{V}\varphi )\circ \xi
),\forall V\in \Im _{0}^{1}(M),(i) \\ 
^{C}\varphi (^{V}A)=^{V}(\varphi (A)),\forall A\in \Im _{q}^{0}(M),\quad (ii)%
\end{array}%
\right.  \label{G3.2}
\end{equation}%
which characterize $^{C}\varphi $, where $\varphi (A)\in \Im _{q}^{0}(M)$.
Remark that $^{V}((L_{V}\varphi )\circ \xi )$ is a vector field on $%
T_{q}^{0}(M)$ and locally expressed by 
\begin{equation*}
^{V}((L_{V}\varphi )\circ \xi )=\left( 
\begin{array}{c}
0 \\ 
(L_{V}\varphi )_{i_{1}}^{j}\xi _{ji_{2}\cdots i_{q}}%
\end{array}%
\right)
\end{equation*}%
with respect to the adapted $(B,C)$-frame, where $\xi _{i_{1}\cdots i_{q}}$
are local components of $\xi $ in $M$ \cite{Magden}.

\begin{theorem}
\ Let $M$ be an almost complex manifold with an almost complex structure $%
\varphi $. Then, the complete lift $^{C}\varphi \in \Im
_{1}^{1}(T_{q}^{0}(M)),$when restricted to the pure cross-section determined
by an almost analytic tensor $\xi $ on $M$, is an almost complex structure.
\end{theorem}

\begin{proof}
If $V\in \Im _{0}^{1}(M)$ and $A\in \Im _{q}^{0}(M)$, in view of the
equations $(i)$ and $(ii)$ of (\ref{G3.2}), we have%
\begin{equation}
(^{C}\varphi )^{2}(^{C}V)=^{C}(\varphi ^{2})(^{C}V)+^{V}(N_{\varphi }\circ
\xi )(^{C}V)  \label{G3.3}
\end{equation}%
and%
\begin{equation}
(^{C}\varphi )^{2}(^{V}A)=^{C}(\varphi ^{2})(^{V}A),  \label{G3.4}
\end{equation}%
where $N_{\varphi ,X}(Y)=(L_{\varphi X}\varphi -\varphi (L_{X}\varphi ))(Y)=%
\left[ \varphi X,\varphi Y\right] -\varphi \left[ X,\varphi Y\right]
-\varphi \left[ \varphi X,Y\right] +\varphi ^{2}\left[ X,Y\right]
=N_{\varphi }(X,Y)$ is nothing but the Nijenhuis tensor constructed by $%
\varphi $.

Let $\varphi \in \Im _{1}^{1}(M)$ be an almost complex structure and $\xi
\in \Im _{q}^{0}(M)$ be a pure tensor with respect to $\varphi .$ If $(\Phi
_{\varphi }\xi )=0$, the pure tensor $\xi $ is called an almost analytic $%
(0,q)-$tensor. In \cite{Salimov,Muto,Koto}, it is proved that $\xi \circ
\varphi \in \Im _{q}^{0}(M)$ is an almost analytic tensor if and only if $%
\xi \in \Im _{q}^{0}(M)$ is an almost analytic tensor. Moreover if $\xi \in
\Im _{q}^{0}(M)$ is an almost analytic tensor, then $N_{\varphi }\circ \xi
=0 $. When restricted to the pure cross-section determined by an almost
analytic tensor $\xi $ on $M$, from (\ref{G3.3}), (\ref{G3.4}) and linearity
of the complete lift, we have

\begin{equation*}
(^{C}\varphi )^{2}=\text{ }^{C}(\varphi ^{2})=\text{ }%
^{C}(-I_{M})=-I_{T_{q}^{0}(M)}.
\end{equation*}%
This completes the proof.
\end{proof}

\section{Complete lift of a symmetric affine connection on a cross-section
in the $(0,q)$-tensor bundle}

We now assume that $\nabla $ is an affine connection (with zero torsion) on $%
M$. Let $\Gamma _{ij}^{h}$\ be components of $\nabla .$ The complete lift $%
^{C}\nabla $ of $\nabla $ to $T_{q}^{0}(M)$ has components $^{C}\Gamma
_{MS}^{I}$ such that%
\begin{eqnarray}
^{C}\Gamma _{ms}^{i} &=&\Gamma _{ms}^{i},\text{ }^{C}\Gamma _{\overline{m}%
s}^{i}=^{C}\Gamma _{m\overline{s}}^{i}=^{C}\Gamma _{\overline{m}\overline{s}%
}^{i}=^{C}\Gamma _{\overline{m}\overline{s}}^{\overline{i}}=0,  \label{G4.1}
\\
^{C}\Gamma _{m\overline{s}}^{\overline{i}} &=&-\dsum\limits_{c=1}^{q}\Gamma
_{mi_{c}}^{s_{c}}\delta _{i_{1}}^{s_{1}}...\delta _{i_{c-1}}^{s_{c-1}}\delta
_{i_{c+1}}^{s_{c+1}}...\delta _{i_{q}}^{s_{q}},  \notag \\
^{C}\Gamma _{\overline{m}s}^{\overline{i}} &=&-\dsum\limits_{c=1}^{q}\Gamma
_{si_{c}}^{m_{c}}\delta _{i_{1}}^{m_{1}}...\delta _{i_{c-1}}^{m_{c-1}}\delta
_{i_{c+1}}^{m_{c+1}}...\delta _{i_{q}}^{m_{q}},  \notag \\
^{C}\Gamma _{ms}^{\overline{i}} &=&\dsum\limits_{c=1}^{q}(-\partial
_{m}\Gamma _{si_{c}}^{a}+\Gamma _{mi_{c}}^{r}\Gamma _{sr}^{a}+\Gamma
_{ms}^{r}\Gamma _{ri_{c}}^{a})t_{i_{1}...i_{c-1}ai_{c+1}...i_{q}}  \notag \\
&&+\frac{1}{2}\dsum\limits_{b=1}^{q}\dsum\limits_{c=1}^{q}(\Gamma
_{mi_{c}}^{l}\Gamma _{si_{b}}^{r}+\Gamma _{mi_{b}}^{l}\Gamma
_{si_{c}}^{r})t_{i_{1}...i_{b-1}ri_{b+1}...i_{c-1}li_{c+1}...i_{q}}  \notag
\\
&&+\sum_{d=1}^{q}t_{i_{1}...l...i_{q}}R_{i_{d}km}^{\text{ \ \ \ \ \ \ }l} 
\notag
\end{eqnarray}%
with respect to the natural frame in $T_{q}^{0}(M)$, where $\delta _{j}^{i}-$%
Kronecker delta and $R_{ikm}^{\text{ \ \ \ \ \ \ }l}$ is components of the
curvature tensor $R$ of $\nabla $ \cite{Magden1}.

We now study the affine connection induced from $^{C}\nabla $ on the
cross-section $\sigma _{\xi }(M)$ determined by the $(0,q)-$tensor field $%
\xi $ in $M$ with respect to the adapted $(B,C)$-frame of $\sigma _{\xi
}(M). $ The vector fields $C_{\overline{j}}$ given by (\ref{G2.5}) are
linearly independent and not tangent to $\sigma _{\xi }(M)$. We take the
vector fields $C_{\overline{j}}$ as normals to the cross-section $\sigma
_{\xi }(M)$ and define an affine connection $\widetilde{\nabla }$ induced on
the cross-section. The affine connection $\widetilde{\nabla }$ induced $%
\sigma _{\xi }(M)$ from the complete lift $^{C}\nabla $ of a symmetric
affine connection $\nabla $ in $M$ has components of the form%
\begin{equation}
\widetilde{\Gamma }_{ji}^{h}=(\partial _{j}B_{i}^{\text{ \ \ }A}+^{C}\Gamma
_{CB}^{A}B_{j}^{\text{ \ }C}B_{i}^{\text{ \ \ }B})B_{\text{ \ }A}^{h},
\label{G4.2}
\end{equation}%
where $B_{\text{ \ }A}^{h}$ are defined by%
\begin{equation*}
(B_{\text{ \ }A}^{h},C_{\text{ \ }A}^{h})=(B_{i}^{\text{ \ \ }A},C_{i}^{%
\text{ \ \ }A})^{-1}
\end{equation*}%
and thus%
\begin{equation}
B_{\text{ \ }A}^{h}=(\delta _{i}^{h},0)\text{, \ \ }C_{\text{ \ }%
A}^{h}=(-\partial _{j}\xi _{k_{1}...k_{q}},\delta _{k_{1}}^{j_{1}}...\delta
_{k_{q}}^{j_{q}}).  \label{G4.3}
\end{equation}

Substituting (\ref{G4.1}), (\ref{G2.4}), (\ref{G2.5}) and (\ref{G4.3}) in (%
\ref{G4.2}), we get%
\begin{equation*}
\widetilde{\Gamma }_{ji}^{h}=\Gamma _{ji}^{h},
\end{equation*}%
where $\Gamma _{ji}^{h}$ are components of $\nabla $ in $M.$

From (\ref{G4.2}), we see that the quantity%
\begin{equation}
\partial _{j}B_{i}^{\text{ \ \ }A}+^{C}\Gamma _{CB}^{A}B_{j}^{\text{ \ }%
C}B_{i}^{\text{ \ \ }B}-\Gamma _{ji}^{h}B_{h}^{\text{ \ \ }A}  \label{G4.4}
\end{equation}%
is a linear combination of the vectors $C_{\overline{i}}^{\text{ \ \ }A}$.
To find the coefficients, we put $A=\overline{h}$ in (\ref{G4.4}) and find%
\begin{equation*}
\nabla _{j}\nabla _{i}\xi _{h_{1}...h_{q}}+\dsum\limits_{\lambda =1}^{q}\xi
_{h_{1}...l...h_{q}}R_{h_{\lambda }ij}^{\text{ \ \ \ \ \ \ }l}.
\end{equation*}%
Hence, representing (\ref{G4.4}) by $\widetilde{\nabla }_{j}B_{i}^{\text{ \
\ }A}$, we obtain%
\begin{equation}
\widetilde{\nabla }_{j}B_{i}^{\text{ \ \ }A}=(\nabla _{j}\nabla _{i}\xi
_{h_{1}...h_{q}}+\dsum\limits_{\lambda =1}^{q}\xi
_{h_{1}...l...h_{q}}R_{h_{\lambda }ij}^{\text{ \ \ \ \ \ \ }l})C_{\overline{h%
}}^{\text{ \ \ \ }A}.  \label{G4.5}
\end{equation}%
The last equation is nothing but the equation of Gauss for the cross-section 
$\sigma _{\xi }(M)$ determined by $\xi _{h_{1}...h_{q}}$. Hence, we have the
following proposition.

\begin{proposition}
The cross-section $\sigma _{\xi }(M)$ in $T_{q}^{0}(M)$\ determined by a $%
(0,q)$ tensor $\xi $ in $M$ with symmetric affine connection $\nabla $ is
totally geodesic if and only if $\xi $ satisfies%
\begin{equation*}
\nabla _{j}\nabla _{i}\xi _{h_{1}...h_{q}}+\dsum\limits_{\lambda =1}^{q}\xi
_{h_{1}...l...h_{q}}R_{h_{\lambda }ij}^{\text{ \ \ \ \ \ \ }l}=0.
\end{equation*}
\end{proposition}

Now, let us apply the operator $\widetilde{\nabla }_{k}$ to (\ref{G4.5}), we
have%
\begin{equation}
\widetilde{\nabla }_{k}\widetilde{\nabla }_{j}B_{i}^{\text{ \ \ }A}=\nabla
_{k}(\nabla _{j}\nabla _{i}\xi _{h_{1}...h_{q}}+\dsum\limits_{\lambda
=1}^{q}\xi _{h_{1}...l...h_{q}}R_{h_{\lambda }ij}^{\text{ \ \ \ \ \ \ }l})C_{%
\overline{h}}^{\text{ \ \ \ }A}.  \label{G4.6}
\end{equation}%
Recalling that%
\begin{equation*}
\widetilde{\nabla }_{k}\widetilde{\nabla }_{j}B_{i}^{\text{ \ \ }A}-%
\widetilde{\nabla }_{j}\widetilde{\nabla }_{k}B_{i}^{\text{ \ \ }A}=%
\widetilde{R}_{DCB}^{\text{ \ \ \ \ \ \ \ }A}B_{k}^{\text{ \ \ }D}B_{j}^{%
\text{ \ \ }C}B_{i}^{\text{ \ \ }B}-R_{kji}^{\text{ \ \ \ \ }h}B_{h}^{\text{
\ \ }A},
\end{equation*}%
and using the Ricci identity for a tensor field of type $(0,q)$, from (\ref%
{G4.6}) we get%
\begin{eqnarray*}
&&\widetilde{R}_{DCB}^{\text{ \ \ \ \ \ \ \ }A}B_{k}^{\text{ \ \ }D}B_{j}^{%
\text{ \ \ }C}B_{i}^{\text{ \ \ }B}-R_{kji}^{\text{ \ \ \ \ }h}B_{h}^{\text{
\ \ }A} \\
&=&[\dsum\limits_{\lambda =1}^{q}(\nabla _{k}R_{h_{\lambda }ij}^{\text{ \ \
\ \ \ }l}-\nabla _{j}R_{h_{\lambda }ik}^{\text{ \ \ \ \ \ }l})\xi
_{h_{1}...l...h_{q}}-R_{kji}^{\text{ \ \ \ \ }l}\nabla _{l}\xi
_{h_{1}...h_{q}} \\
&&-\dsum\limits_{\lambda =1}^{q}R_{kjh_{\lambda }}^{\text{ \ \ \ \ \ }%
l}\nabla _{i}\xi _{h_{1}...l...h_{q}}+\dsum\limits_{\lambda
=1}^{q}R_{h_{\lambda }ij}^{\text{ \ \ \ \ \ }l}\nabla _{k}\xi
_{h_{1}...l...h_{q}}-\dsum\limits_{\lambda =1}^{q}R_{h_{\lambda }ik}^{\text{
\ \ \ \ \ }l}\nabla _{j}\xi _{h_{1}...l...h_{q}}]C_{\overline{h}}^{\text{ \
\ }A}.
\end{eqnarray*}%
Thus we have the result below.

\begin{proposition}
$\widetilde{R}_{DCB}^{\text{ \ \ \ \ \ \ \ }A}B_{k}^{\text{ \ \ }D}B_{j}^{%
\text{ \ \ }C}B_{i}^{\text{ \ \ }B}$ is tangent to the cross-section $\sigma
_{\xi }(M)$ if and only if 
\begin{eqnarray*}
&&\dsum\limits_{\lambda =1}^{q}(\nabla _{k}R_{h_{\lambda }ij}^{\text{ \ \ \
\ \ }l}-\nabla _{j}R_{h_{\lambda }ik}^{\text{ \ \ \ \ \ }l})\xi
_{h_{1}...l...h_{q}} \\
&=&R_{kji}^{\text{ \ \ \ \ }l}\nabla _{l}\xi
_{h_{1}...h_{q}}+\dsum\limits_{\lambda =1}^{q}R_{kjh_{\lambda }}^{\text{ \ \
\ \ \ }l}\nabla _{i}\xi _{h_{1}...l...h_{q}}-\dsum\limits_{\lambda
=1}^{q}R_{h_{\lambda }ij}^{\text{ \ \ \ \ \ }l}\nabla _{k}\xi
_{h_{1}...l...h_{q}} \\
&&+\dsum\limits_{\lambda =1}^{q}R_{h_{\lambda }ik}^{\text{ \ \ \ \ \ }%
l}\nabla _{j}\xi _{h_{1}...l...h_{q}}.
\end{eqnarray*}
\end{proposition}

\end{document}